\documentclass[10pt]{article}
\font\smallit=cmti10
\font\smalltt=cmtt10

\usepackage{amssymb,latexsym,amsmath,epsfig,amsthm} 

\usepackage{amsmath,amsfonts}
\usepackage{graphicx}
\usepackage{xcolor}
\usepackage{float}
\usepackage{hyperref}

\makeatletter

\renewcommand\section{\@startsection {section}{1}{\z@}
{-30pt \@plus -1ex \@minus -.2ex}
{2.3ex \@plus.2ex}
{\normalfont\normalsize\bfseries\boldmath}}

\renewcommand\subsection{\@startsection{subsection}{2}{\z@}
{-3.25ex\@plus -1ex \@minus -.2ex}
{1.5ex \@plus .2ex}
{\normalfont\normalsize\bfseries\boldmath}}

\renewcommand{\@seccntformat}[1]{\csname the#1\endcsname. }

\makeatother

\newtheorem{thm}{Theorem}

\newtheorem{lem}{Lemma}[section]
\newtheorem{prop}[lem]{Proposition}
\newtheorem{cor}[lem]{Corollary}

\theoremstyle{definition}
\newtheorem{dfn}[lem]{Definition}
\newtheorem{conj}[lem]{Conjecture}

\newtheorem{example}[lem]{Example}

\numberwithin{equation}{section}

\newcommand{\fl}[1]{\lfloor #1\rfloor}

\renewcommand{\arraystretch}{1.4}

\newcommand{\sI}{s_I}
\newcommand{\sII}{s_{II}}

\newcommand{\rand}{rand}
\newcommand{\opt}{opt}

\newcommand{\psIsII}{p_n^{s_I,s_{II}}}
\newcommand{\poptopt}{p_n^{\opt,\opt}}
\newcommand{\prandopt}{p_n^{\rand,\opt}}
\newcommand{\poptrand}{p_n^{\opt,\rand}}

\begin{document}

\begin{center}
\uppercase{\bf The number of optimal strategies in the 
Penney-Ante Game}
\vskip 20pt
{\bf Reed Phillips}\\
{\smallit Department of Mathematics, Rose-Hulman Institute of
Technology, Terre Haute, Indiana, USA}\\
{\tt phillirc@rose-hulman.edu}
\\
\vskip 10pt
{\bf A.J. Hildebrand}\\
{\smallit Department of Mathematics, University of Illinois, 
Urbana, Illinois, USA}\\
{\tt ajh@illinois.edu}
\end{center}
\vskip 20pt

\centerline{\smallit Received: , Revised: , Accepted: , Published: } %
\vskip 30pt

\centerline{\bf Abstract}

\noindent
In the Penney-Ante game, Player I chooses a head/tail string of a
predetermined length $n\ge3$. Player II, upon seeing Player I's choice,
chooses another head/tail string of the same length. A coin is then
tossed repeatedly and the player whose string appears first in the
resulting head/tail sequence wins the game. The Penney-Ante game has gained notoriety as a source of counterintuitive probabilities and
nontransitivity phenomena. For example, Player II can always choose a
string that beats the choice of Player I in the sense of being more
likely to appear first in a random head/tail sequence.

It is known that Player II has a unique optimal strategy
that maximizes her winning chances in this game.  On the other hand, for Player
I there exist multiple equivalent optimal strategies.  In this paper we
investigate the number, $c_n$, of optimal strategies for Player I, i.e.,
the number of head/tail strings of length $n$ that maximize the winning
probability for Player I assuming optimal play by Player
II.  We derive a recurrence relation for $c_n$ and use this to
obtain a sharp asymptotic estimate for $c_n$. In
particular, we show that, as $n\to\infty$, a fixed proportion $\alpha \approx
0.04062\dots$ of the $2^n$ head/tail strings of length $n$ are optimal
from Player I's perspective.

\pagestyle{myheadings}
\markright{\smalltt INTEGERS: 21 (2021)\hfill}
\thispagestyle{empty}
\baselineskip=12.875pt
\vskip 30pt

\section{Introduction and Statement of Results}
\label{sec:introduction}

\paragraph{The Penney-Ante Game.}
Penney-Ante is a coin-flipping game created some fifty years by Walter
Penney \cite{penney1969} and popularized by Martin Gardner
\cite{gardner1974}, who called it
``one of the most incredible of all nontransitive betting situations.'' 
The game involves two players, I and II, 
and in its usual formulation proceeds as follows:
\begin{quote}
\emph{Player I begins by choosing a head/tail string of a predetermined length
$n\ge3$.
Player II, upon seeing Player I's choice, chooses another
head/tail string of the same length $n$.
A coin is then tossed repeatedly until one of the two strings
chosen by the players appears. The player whose string appears first
wins the game.
}
\end{quote}

The Penney-Ante game is a source of many counterintuitive probabilities and
examples of nontransitivity.  Perhaps the most striking
feature of this game is that, given \emph{any} string of length at
least $3$, there always exists another string of the same length that beats the
given string in the sense of being more likely to appear first in an
infinite sequence of coin tosses.  As a consequence, Player II always
has the advantage in the Penney-Ante game as she can choose a string
that beats the string selected by Player I.

Table \ref{table:prob-matrix}, taken from Gardner \cite{gardner1974},
shows the pairwise winning probabilities in the Penney-Ante game
with strings of length $n=3$.  The entry indexed by row string $B$ and
column string $A$ represents the probability that $B$ appears before $A$ in a
random head/tail sequence, i.e., the probability that a player with
string $B$ wins over a player with string $A$.

\begin{table}[H]
	\begin{center}
	\begin{tabular}{|c||c|c|c|c|c|c|c|c|}
		\hline
	 	B $\backslash$ A 
		& HHH & HHT & HTH & HTT & THH & THT & TTH & TTT 
		\\
		\hline\hline
		HHH& & 1/2 & 2/5 & 2/5 & 1/8 & 5/12 & 3/10 & 1/2
		\\ \hline
		HHT& 1/2 & & 2/3 & 2/3 & 1/4 & 5/8 & 1/2 & 7/10
		\\ \hline
		HTH& 3/5 & 1/3 & & 1/2 & 1/2 & 1/2 & 3/8 & 7/12 
		\\ \hline
		HTT& 3/5 & 1/3 & 1/2 & & 1/2 & 1/2 & 3/4 & 7/8 
 		\\ \hline
		THH& 7/8 & 3/4 & 1/2 & 1/2 & & 1/2 & 1/3 & 3/5 
 		\\ \hline
		THT& 7/12& 3/8 & 1/2 & 1/2 & 1/2 & & 1/3 &  3/5 
 		\\ \hline
		TTH& 7/10& 1/2 & 5/8 & 1/4 & 2/3 & 2/3 & &  1/2 
 		\\ \hline
		TTT& 1/2 & 3/10& 5/12& 1/8 & 2/5 & 2/5 & 1/2 & 
		\\
		\hline
	\end{tabular}
	\end{center}
	\caption{Pairwise winning probabilities
	in the Penney-Ante game with strings of length~$3$.
	}
	\label{table:prob-matrix}
\end{table}

The probabilities in Table \ref{table:prob-matrix} can be computed by
elementary probabilistic arguments. For example, the fact that the string
$THH$ ``beats'' the string $HHH$  with probability $7/8$  can be seen by
observing that the only way for the string $HHH$ to appear \emph{before}
the string $THH$ in an infinite head/tail sequence (and thus win the
Penney-Ante game) is when the sequence \emph{starts out} with the string
$HHH$, an event that occurs with probability $1/8$.  

For strings of general length $n$, John Conway
(see Gardner \cite{gardner1974}) gave an ingenious algorithm for
computing the pairwise winning probabilities. The algorithm involves the
so-called \emph{Conway numbers}, which are positive integers associated
to any pair of finite head/tail strings and which measure the amount of
overlap between these two strings.  We will describe Conway's algorithm
in Section \ref{sec:conway}.

The Penney-Ante game and related questions have been studied in the
literature using a variety of methods including combinatorial approaches and
generating functions \cite{csirik1992,felix2006, guibas-odlyzko1981,
noonan-zeilberger1999}, martingales \cite{li1980}, Markov
chains \cite{blom-thorburn1982,chen-zane1979}, renewal theory
\cite{breen1985}, and gambling models \cite{pozdnyakov-kulldorff2006}. 
Some of the deepest work on the Penney-Ante game is due to Guibas and
Odlyzko \cite{guibas-odlyzko1981}. Motivated by applications to string
search algorithms, these authors framed the Penney-Ante game as a
problem in the theory of combinatorics of words. 
Using
a generating function approach, they considered the general problem of
counting strings of a given length over a finite alphabet that end in a
specified string and that do not contain any string from a given set of
``forbidden'' strings as substring.  The Penney-Ante game can be viewed
as a special case of this problem corresponding to sequences over the
two letter alphabet $\{H,T\}$ that end in a specified string $A$ and do not contain
another specified string $B$ of the same length as $A$. 

\paragraph{Optimal strategy for Player II.}
Perhaps the most natural question in the Penney-Ante game is the
following: 

\begin{quote}
\emph{Given a string selected by Player I, how should Player II choose
her string to maximize the probability of winning the Penney-Ante game?
In other words, given a string $A$, what is the ``best response
string'' $B$ to this  string?}
\end{quote}

For small values of $n$, 
such best response strings can be determined directly
by inspecting pairwise probability
tables such as Table \ref{table:prob-matrix}.  For example,
from the first column in Table \ref{table:prob-matrix} we see that the
maximal winning probability against the string $HHH$ is $7/8$, and that
$THH$ is the only string achieving this probability. Thus, $THH$ is the
unique best response string against the string $HHH$. 
Tables \ref{table:optimal-responses3} and \ref{table:optimal-responses4}
below show the best response strings for all strings of length $3$ and
$4$.  In each case, the string $B$ listed in the second column is the
\emph{unique} string that maximizes the winning probability for Player II
against the string $A$ in the first column, and the probability in the third column is
the corresponding maximal winning probability. 

\begin{table}[H]
	\begin{center}
		\begin{tabular}{|c|c|c|}
			\hline
			A & B & Probability \\
			\hline
			\hline
			HHH & THH & $7/8$
			\\ \hline
			HHT & THH & $3/4$
			\\ \hline
			HTH & HHT & $2/3$
			\\ \hline
			HTT & HHT & $2/3$
			\\ \hline
	\end{tabular}
	\hspace{2em}
		\begin{tabular}{|c|c|c|}
			\hline
			A & B & Probability \\
			\hline
			\hline
			THH & TTH & $2/3$
			\\ \hline
			THT & TTH & $2/3$
			\\ \hline
			TTH & HTT & $3/4$
			\\ \hline
			TTT & HTT & $7/8$
			\\
			\hline
	\end{tabular}
	\end{center}
	\caption{Best response strings for strings
	of length~$3$.}
	\label{table:optimal-responses3}
\end{table}

\begin{table}[H]
	\begin{center}
		\begin{tabular}{|c|c|c|}
			\hline
			A & B & Probability \\
			\hline
			\hline
HHHH & THHH & $15/16$\\ \hline
HHHT & THHH & $7/8$ \\ \hline
HHTH & HHHT & $2/3$\\ \hline
HHTT & HHHT & $2/3$\\ \hline
HTHH & THTH & $9/14$\\ \hline
HTHT & HHTH & $5/7$\\ \hline
HTTH & HHTT & $2/3$\\ \hline
HTTT & HHTT & $2/3$\\ \hline
	\end{tabular}
	\hspace{2em}
		\begin{tabular}{|c|c|c|}
			\hline
			A & B & Probability \\
			\hline
			\hline
THHH & TTHH & $2/3$\\ \hline
THHT & TTHH & $2/3$\\ \hline 
THTH & TTHT & $5/7$\\ \hline
THTT & HTHT & $9/14$\\ \hline
TTHH & TTTH & $2/3$\\ \hline
TTHT & TTTH & $2/3$\\ \hline
TTTH & HTTT & $7/8$\\ \hline
TTTT & HTTT & $15/16$\\ \hline
	\end{tabular}
	\end{center}
	\caption{Best response strings for strings
	of length~$4$.}
	\label{table:optimal-responses4}
\end{table}

For strings of general length $n\ge3$, Guibas and Odlyzko
\cite{guibas-odlyzko1981} gave a simple algorithm to determine
the best response string  up to the choice of a single initial letter:
Namely, given a string $A$ of length $n$, they showed that the
best response string is of the form $HA'$ or $TA'$, where
$A'$ is the string
consisting of the first $n-1$ letters of $A$. Thus, for example, the best
response string to \fbox{$HHTHT$}$\,T$ is of the form $H\,$\fbox{$HHTHT$} or
$T$\,\fbox{$HHTHT$}.
Guibas and Odlyzko went on to conjecture that, among the two possible
forms of the best response string, there is always one that performs
strictly better than the other in the Penney-Ante game. This conjecture
was proved by Csirik \cite{csirik1992}. Felix \cite{felix2006} gave
another proof of this result and also provided an algorithm to determine
which of the two candidates for the best response string identified by
Guibas and Odlyzko is the true best response. 

It follows from these results that Player II always has a unique optimal
response strategy in the Penney-Ante game.

\paragraph{Optimal strategies for Player I.}
We can ask similarly for optimal strategies from Player I's
perspective:  

\begin{quote}
\emph{Which string should Player I choose to maximize his chances of winning
the Penney-Ante game assuming optimal play by Player II?  Equivalently,  
which string $A$ \emph{minimizes} the probability that the best response
string $B$ to $A$ wins the game?} 
\end{quote}

As it turns out, the answer to this question is markedly different from
that about Player II's optimal strategy.  While Player II always has a
unique optimal strategy, Player I has many optimal strategies that are
all equivalent in the sense of securing the same winning odds for Player
I.  Indeed, Table \ref{table:optimal-responses3} shows that, when $n=3$,
the \emph{smallest} winning probability for Player II under optimal play
is $2/3$, and that this probability is achieved when Player I chooses one
the four strings $HTH$, $HTT$, $THH$, and $THT$.  Thus, these four strings all
tie as optimal strategies for Player I.  Similarly, from Table
\ref{table:optimal-responses4} we see that, when $n=4$, Player I has
exactly two optimal strategies, given by the strings $HTHH$ and $THTT$. 

For strings of length $n\ge 5$, Csirik \cite{csirik1992}
characterized all optimal strings  for Player I in terms of Conway
numbers; see
Proposition \ref{prop:csirik-theorem3} below.

\paragraph{The number of optimal strategies for Player I.}
As mentioned, Player I has in general many optimal strategies, each
yielding the same maximal winning probability. This raises the
following question:
\begin{quote}
\emph{How many optimal strategies are there for Player I? That is, given
$n\ge 3$, how many strings of length $n$ are there that maximize
the winning probability for Player I in the Penney-Ante game 
assuming optimal play by Player II?
}
\end{quote}
This question will be the main focus of this paper.  Letting $c_n$ denote the
number of ``optimal'' strings for Player I described in this question,
we are interested in determining the behavior and properties of 
the sequence $\{c_n\}$.

As mentioned above, for $n=3$ there are four strings that tie as optimal
strategies, while for $n=4$ there are two such strings. Thus we have
$c_3=4$ and $c_4=2$.  Table \ref{table:cn-initial-values} provides
further values of $c_n$. 

\begin{table}[H]
\begin{center}
\renewcommand{\arraystretch}{1.3}
\begin{tabular}{|c|c|c|c|c|c|c|c|c|c|c|c|c|c|}
\hline
$n$& 3 & 4 & 5 & 6 & 7 & 8 & 9&10&11&12&13&14&15
\\
\hline
$c_n$ & 4 & 2 & 2 & 2 & 6 & 10 & 22 & 42 & 86 & 166 & 338 & 666 & 1342
\\
\hline
\end{tabular}
\caption{Values of $c_n$, the number of optimal strategies for Player I 
in the Penney-Ante game with strings of length $n$.}
\label{table:cn-initial-values}
\end{center}
\end{table}

The sequence $\{c_n\}$ shown in Table \ref{table:cn-initial-values}
does not seem to have a closed form, and the sequence is also not
listed in the \emph{On-Line Encyclopedia of Integer Sequences}
\cite{oeis}. Thus, it is likely that this sequence has not
occurred before in some other context.  Our main goal in this paper is to
gain a more complete understanding of this sequence, its properties, and
its asymptotic behavior.

We first use Csirik's characterization of the optimal
strings for Player I to derive a recurrence relation satisfied by $c_n$.


\begin{thm}
	\label{thm:cn-recurrence}
The number $c_n$ of optimal strings for Player I satisfies 
\begin{align}
    \label{eq:cn-recurrence}
	c_n=2c_{n-1}-(-1)^nc_{\fl{n/2}+1}\quad (n\ge 6).
\end{align}
\end{thm}

We next use this relation to determine the asymptotic behavior of $c_n$.

\begin{thm}
        \label{thm:cn-asymptotic}
        As $n\to\infty$, we have
\begin{equation}
\label{eq:cn-asymptotic}
                c_n=
                \begin{cases}
                \alpha 2^n + O\left(2^{n/4}\right) 
                & \text{if $n$ is even,}
                \\
                \alpha \left(2^n + 2^{\fl{n/2}+1}\right)+ O(2^{n/4})
                & \text{if $n$ is odd,}
\end{cases}
        \end{equation}
where $\alpha$ is a positive constant with approximate numerical value
\begin{equation}
\label{eq:alpha-values}
                \alpha=0.040602\dots
\end{equation}
In particular, as $n\to\infty$, a fixed proportion $\alpha$
of all $2^n$ head/tail strings of length $n$ represent optimal
strategies for Player I. 
\end{thm}

The asymptotic estimate \eqref{eq:cn-asymptotic}
has an interesting interpretation
in terms of the binary representations of $c_n$ and $\alpha$:
Letting $\alpha=0.\alpha_1\alpha_2\dots$ denote the binary expansion of
$\alpha$, we have
$\alpha 2^n=\alpha_1\alpha_2\dots \alpha_n.\alpha_{n+1}\dots$ 
Thus the integer part of $\alpha 2^n$, the main term in the estimate
\eqref{eq:cn-asymptotic}, 
consists of the first $n$ binary digits
of the constant $\alpha$.  The other terms on the right 
of \eqref{eq:cn-asymptotic} are of order
at most $O(2^{n/2})$ and thus affect only the last $n/2$ or so
binary bits of $c_n$. Consequently, approximately the first half of the binary
digits of $c_n$ coincide with the binary digits of $\alpha$. 
This behavior is illustrated in Table \ref{table:cn-binary-values},
which shows the binary expansions of the values $c_n$ for $5\le n\le
25$.

\begin{table}[H]
\begin{center}
\renewcommand{\arraystretch}{1.0}
\addtolength{\tabcolsep}{3pt}
\begin{tabular}{|c|r|l|}
\hline
$n$ & $c_n$ & $c_n$ in Binary
\\
\hline
5  &  2  &  10  \\
6  &  2  &  10  \\
7  &  6  &  110  \\
8  &  10  &  1010  \\
9  &  22  &  10110  \\
10  &  42  &  101010  \\
11  &  86  &  1010110  \\
12  &  166  &  10100110  \\
13  &  338  &  101010010  \\
14  &  666  &  1010011010  \\
15  &  1342  &  10100111110  \\
16  &  2662  &  101001100110  \\
17  &  5346  &  1010011100010  \\
18  &  10650  &  10100110011010  \\
19  &  21342  &  101001101011110  \\
20  &  42598  &  1010011001100110  \\
21  &  85282  &  10100110100100010  \\
22  &  170398  &  101001100110011110  \\
23  &  340962  &  1010011001111100010  \\
24  &  681586  &  10100110011001110010  \\
25  &  1363510  &  101001100111000110110  \\
\hline
\end{tabular}
\caption{Decimal and binary values of $c_n$.}
\label{table:cn-binary-values}
\end{center}
\end{table}

In fact, we have the following \emph{exact} formula for $\alpha$ in terms
of an infinite series involving the numbers $c_n$.

\begin{thm}
        \label{thm:alpha-series}
	The constant $\alpha$ defined by \eqref{eq:cn-asymptotic} satisfies 
	\begin{equation}
        \label{eq:alpha-series}
        \alpha = \frac{1}{16}-2\sum_{n=4}^\infty \frac{c_n}{4^{n}}.
    \end{equation}
\end{thm}

This series allows for an efficient computation of the constant
$\alpha$.  Indeed, since, by Theorem \ref{thm:cn-asymptotic}, $c_n$ is of order
$O(2^n)$, the terms in the series \eqref{eq:alpha-series} decay at rate
$2^{-n}$ and truncating this series after $n$ terms gives $\alpha$
within an accuracy of order $O(2^{-n})$.

\paragraph{Optimal strategies in the ``flipped'' Penney-Ante game.}
A natural question that does not seem to have received attention in the
literature is whether analogous results hold  in a ``flipped''
Penney-Ante game where the player whose string appears \emph{last} wins the
game.  

Clearly, the odds of one string of length $n$ winning over another such string
in the flipped Penney-Ante game are the reciprocals of the odds for the standard
Penney-Ante game.  Similarly, the matrix of pairwise winning
probabilities for the  flipped game is the transpose of the corresponding
matrix for the standard game.  Because of this symmetry, one might expect
that the properties of the flipped game are largely analogous to those of
the standard game.  Surprisingly, this is not the case.  We will show:

\begin{thm}
	\label{thm:flipped-optimal-moves} 
Let $n\ge 3$, and consider the flipped  Penney-Ante game on strings of
length $n$. Then there are exactly two optimal strategies for Player I,
namely the strings $HH\dots H$ and $TT\dots T$ consisting of $n$ heads or $n$ tails.
Under these strategies, Player I wins with probability $1/2$.
\end{thm}

\paragraph{Outline of the paper.} In Section \ref{sec:conway} we describe
Conway's algorithm for computing pairwise winning probabilities in the
Penney-Ante game, and we prove some basic properties of
the Conway numbers on which this algorithm is based.
In Sections \ref{sec:theorem1-proof}--\ref{sec:theorem4-proof}, we prove
our main results, Theorems
\ref{thm:cn-recurrence}--\ref{thm:flipped-optimal-moves}.
We conclude in Section \ref{sec:further-results} by presenting some
open problems and conjectures related to these results.  In particular, we consider the question of how much of a penalty each player incurs by playing randomly (i.e., choosing one of the $2^n$ strings at random) instead of optimally.

\section{Conway Numbers and Conway's Algorithm}
\label{sec:conway}

In this section we describe Conway's algorithm for computing pairwise
winning probabilities in the Penney-Ante game and prove some auxiliary
results. 

In what follows all strings are assumed to be finite binary strings 
over the symbols $H$ and $T$.  We use uppercase letters 
to denote such strings and lowercase letters to denote
the individual bits in these strings;  for example, $A=a_1a_2\dots a_n$ 
denotes a generic string of length $n$ over the alphabet $\{H,T\}$.

Conway's algorithm is based on the concept of \emph{Conway numbers},
which are defined as follows (see, e.g., Gardner \cite{gardner1974}, 
or Guibas and Odlyzko \cite{guibas-odlyzko1981}).

\begin{dfn}[Conway numbers]
\label{def:conway-number}
Let $A=a_1\dots a_n$ and $B=b_1\dots b_n$ be strings of length $n$.
\begin{itemize}
\item[(i)] The \emph{Conway number}, or \emph{correlation}, of $A$ and $B$ 
is the nonnegative integer defined by
\begin{equation}
\label{eq:conway-number-def}
C(A,B) =\sum_{i=1}^n \delta_i 2^{n-i},
\end{equation}
where 
\begin{equation}
\label{eq:delta-i-def}
\delta_i=\begin{cases}
1 &\text{if $a_{i+j}=b_{1+j}$ for $j=0,\dots n-i$,}
\\
0 & \text{otherwise.}
\end{cases}
\end{equation}
In other words, $C(A,B)$ is the number with
binary expansion given by $\delta_1\delta_2\dots\delta_n$, 
where $\delta_i=1$ if the last $n-i+1$ bits of the string
$A$ coincide with the first $n-i+1$ bits of $B$, and $\delta_i=0$
otherwise.

\item[(ii)] 
The \emph{autocorrelation} of $A$ is defined as 
the correlation of $A$ with itself, i.e., as the Conway number $C(A,A)$.
\end{itemize}
\end{dfn}

Note that, by \eqref{eq:delta-i-def},
the leading bit,  $\delta_1$, in the Conway number $C(A,B)$
is equal to $1$ if and only if the two
strings $A$ and $B$ are equal. It follows that
the Conway number of two different
strings of length $n$ is at most $\sum_{i=2}^n
2^{n-i}=2^{n-1}-1$, while the Conway number of two identical strings of
length $n$ (i.e., the autocorrelation of this string) is at least
$2^{n-1}$ and at most $\sum_{i=1}^n 2^{n-i}=2^n-1$.

Conway numbers can be interpreted as either  binary
strings over $\{0,1\}$ (padded with leading $0$s if necessary so that the
string has length $n$), or as the integers represented by these strings.
In what follows we will use these two interpretations interchangeably.

We illustrate the calculation of Conway numbers with an example.

\begin{example}
Let $A=HHTHT$ and $B=HTHTT$. To calculate the bits $\delta_i$ of the
Conway number $C(A,B)$ first line up the two strings. If they are equal,
write a $1$ under the leading bits of the two strings; 
otherwise write a $0$:
\begin{flushleft}
\addtolength{\tabcolsep}{-5pt}
\hspace{.4\textwidth}
\begin{tabular}{c c c c c c|c}
H & H & T & H & T&& A \\
H & T & H & T & T && B\\
\hline
0 & & & && &\ C(A,B)\\
\end{tabular}
\end{flushleft}
Then repeatedly shift $A$ to the left, make the same comparison on the
overlapping parts of the two strings and write the result (i.e., 
$1$ if these parts match, and $0$ otherwise)
under the leading bits of the overlapping parts:
\begin{center}
\addtolength{\tabcolsep}{-5pt}
\begin{tabular}{r}
\begin{tabular}{c c c c c c c|c}
H & H & T & H & T & && A\\
& H & T & H & T & T && B\\
\hline
0 & 1 & & &&& &\ C(A,B) \\
\end{tabular} \\ \\

\begin{tabular}{c c c c c c c c|c}
H & H & T & H & T & & && A\\
& & H & T & H & T & T  && B\\
\hline
0 & 1 & 0 & &&&&&\ C(A,B) \\
\end{tabular} \\ \\

\begin{tabular}{c c c c c c c c c|c}
H & H & T & H & T & & & && A\\
& & & H & T & H & T & T && B\\
\hline
 0 & 1 & 0 & 1 & &&&&&\ C(A,B)\\
\end{tabular} \\ \\

\begin{tabular}{c c c c c c c c c c |c}
 H & H & T & H & T & & & & && A \\
 & & & & H & T & H & T & T  && B\\
\hline
 0 & 1 & 0 & 1 & 0 &&&&&&\ C(A,B) \\
\end{tabular}
\end{tabular}
\end{center}
At this point another shift would leave no overlap, and the algorithm
terminates. The binary string obtained in the last step is the binary
expansion of the Conway number of $A$ and $B$.
In the above example the final result is the binary string $01010$, 
so the Conway number $C(A,B)$ is $2^3+2^1=10$.
\end{example}

Using the concept of Conway numbers, Conway gave a remarkably simple
formula for computing the pairwise winning odds in the Penney-Ante
game. His result is as follows (see, e.g., Gardner \cite{gardner1974}). 

\begin{prop}
[Conway's Algorithm]
Let $A=a_1a_2...a_n$ and $B=b_1b_2...b_n$ be two distinct 
head/tail strings of length $n$.
Then the odds in favor of string $A$ over string $B$
in the Penney-Ante game are given by
\begin{equation}
\label{eq:conway-algorithm}
\frac{P(\text{$A$ appears before $B$})}
{P(\text{$B$ appears before  $A$})}=
\frac{C(B,B)-C(B,A)}{C(A,A)-C(A,B)}.
\end{equation}
\end{prop}

The following lemma establishes a connection between counts of strings with a
given autocorrelation and counts of \emph{pairs} of strings with a given
Conway number.

\begin{lem}
\label{lem:autocorr-sum}
Let $m$ and $k$ be positive integers with $0\le k\le 2^m-1$.
The number of pairs $(A_1,A_2)$
of strings of length $m$ with Conway number $k$ is
equal to the number of strings $A$ of length $2m$ whose autocorrelation is
congruent to $k$ mod $2^{m}$,  i.e., has a binary representation
that ends in the binary bits of $k$
(padded out to a string of length $m$ if necessary).

Moreover, if $X$ and $Y$ are strings of length less than $m$, the same
conclusion holds under the restrictions that  
  $X$ is a prefix of both $A$ and $A_1$ and 
        $Y$ is a suffix of both $A$ and $A_2$.
\end{lem}

\begin{proof}
By letting $X$ and $Y$ be the empty strings, the first part of the
lemma is seen to be a special case of the second part, so it suffices 
to prove the latter part.

Let $X$ and $Y$ be strings of length less than $m$
and consider a string
$A$ of length $2m$ with $Y$ as a suffix and $X$ as a prefix. 
Write $A=A_1A_2$, where $A_1$ (resp. $A_2$)
is the string consisting of the first $m$ (resp. last $m$)
bits of $A$. Then $A_1$ has prefix $X$ and $A_2$ has suffix $Y$.
    
At the $m$th step of computing the autocorrelation of $A$, the  top copy
of $A$ will have been shifted by exactly $m$ bits to the left so that the
overlapping parts of the two copies of $A$ consist of the substrings
$A_2$ and $A_1$.  From then on, the calculation is the same as that of
the Conway number $C(A_2,A_1)$.  Therefore the last $m$ bits of the
autocorrelation of $A$ are the same as the $m$ bits of the 
Conway number $C(A_2,A_1)$.  It is easy to check 
that the mapping $A\to(A_1,A_2)$ defined in this way yields a bijection
between the following sets:
\begin{itemize}
\item[(I)] Strings $A$ of length $2m$ beginning with $X$
and ending in $Y$ whose autocorrelation ends in a given binary string of
length $m$. 
\item[(II)]
Pairs $(A_1,A_2)$ of strings of length $m$ such that $A_1$ begins with $X$, 
$A_2$ ends with $Y$, and the Conway number $C(A_2,A_1)$ is exactly equal to 
the given binary string. 
\end{itemize}
The claim now follows.
\end{proof}

The next lemma establishes some properties of autocorrelations that we
will need for the proof of Theorem \ref{thm:cn-recurrence}.

\begin{lem}
\label{lem:correctAutocorrelation}
Let $m\ge 2$.
\begin{itemize}
\item[(i)]
The only possible autocorrelation of length $2m$ 
whose last $m$ bits are $00\dots 01$ is the $(2m)$-bit string 
$100\dots 01$. 
In other words, if the 
autocorrelation of a string of length $2m$ is congruent to $1$ modulo
$2^m$, then it must be equal to  $2^{2m}+1$.

\item[(ii)]
The only possible autocorrelations of length $2m+1$ 
whose last $m$ bits are $00\dots 01$ 
are the $(2m+1)$-bit strings 
$100\dots 01$ 
and $1\underbrace{00\dots 0}_{m-1}1\underbrace{0\dots 0}_{m-1}1$.

\item[(iii)]
The only possible autocorrelations of length $2m+2$ 
whose last $m$ bits are $00\dots 01$ 
are the $(2m+2)$-bit strings 
$100\dots 01$ and
$1\underbrace{00\dots 0}_m1\underbrace{0\dots 0}_{m-1}1$.
\end{itemize}
\end{lem}

\begin{proof}
Let $A=a_1a_2\dots a_n$ 
be a string of length $n=2m$ or $n=2m+1$ with autocorrelation ending in
the $m$ bit string $00\dots 01$. The conclusions of parts (i) and (ii)
will follow if we can show that the first $m$ bits of the autocorrelation
of $A$ are $100\dots 0$. 

Since the leading bit  $\delta_1$ of any
autocorrelation  must equal $1$, it suffices to show that the bits
$\delta_i$,  $i=2,3,\dots,m$, must all be $0$.  
We argue by contradiction. Suppose $\delta_k=1$ for some $k$ with $2\le
k\le m$.  Then $a_i=a_{i+k-1}$ for $1\le i\le n-k+1$. 
Iterating this identity yields $a_i=a_{i+q(k-1)}$ for any positive
integer $q$ satisfying $q<n/(k-1)$ and any $i$ with $1\le i\le n-q(k-1)$.  
But this implies $\delta_{q(k-1)+1}=1$ for any $q<n/(k-1)$.
It follows that 
among any $k-1$ consecutive indices $i$ there is at least one such that
$\delta_i=1$. Applying this observation to the set of indices $\{n-k+1,
n-k+2,\dots, n-1\}$ we conclude that $\delta_{n-j}=1$ for
some $j$ with $1\le j\le k-1$. Since $k\le m$, this contradicts the
assumption that the last $m$ bits of the autocorrelation of $A$
are $0\dots 01$. This completes the proof of parts (i) and (ii). 

For the proof of part (iii), assume $A$ is a string of length $n=2m+2$
with autocorrelation ending in the $m$ bits $00\dots 01$. The same
argument as for parts (i) and (ii) yields that the first $m$ bits of the
autocorrelation of $A$ are of the desired form, namely 
$100\dots 0$. Hence, the only bits of the autocorrelation other than the
first and last bit that can possibly be equal to $1$ are $\delta_{m+1}$
and $\delta_{m+2}$. To obtain the desired conclusion we must rule out
the case $\delta_{m+1}=1$.   

Suppose $\delta_{m+1}=1$. Then the above argument yields
$\delta_{1+qm}=1$ for any $q<n/m$.  In particular, it follows that
$\delta_{2m+1}=1$. But then the autocorrelation of $A$ ends in the two bits 
$11$, contradicting the assumptions of the lemma.
This completes the proof.
\end{proof}

We remark that the reasoning employed in this proof can be viewed as a
special case of the \emph{forward propagation rule} of Guibas and Odlyzko
\cite[Theorem 5.1]{guibas-odlyzko1981}.

\section{Proof of Theorem \protect\ref{thm:cn-recurrence}}
\label{sec:theorem1-proof}

Theorem \ref{thm:cn-recurrence} asserts 
that the number $c_n$ of optimal strategies
for Player I satisfies the recurrence
\eqref{eq:cn-recurrence}, i.e., 
\begin{equation}
\label{eq:cn-recurrence-DUP}
	c_n=2c_{n-1}-(-1)^nc_{\fl{n/2}+1}\quad (n\ge 6).
\end{equation}

Our argument is based on Csirik's characterization of optimal strategies 
for Player I, which we state in the following proposition. 

\begin{prop}
	[{Csirik \cite[Corollary 4]{csirik1992}}]
		\label{prop:csirik-theorem3}
Let $n\ge 5$. The optimal strategies for Player I in the Penney-Ante game
with strings of length $n$ are exactly
the strings of the form $A=HTa_3\dots a_{n-3}THH$  or
$A=THa_3\dots a_{n-3}HTT$
such that the  
$(n-1)$-bit prefix of $A$ has autocorrelation $2^{n-2}+1$.
Under these strategies, the probability that Player I wins the game
assuming optimal play by Player II is given by 
\begin{equation}
\label{eq:csirik-winning-probability}
P(\text{Player I wins})
=\frac{2^{n-2}+1}{3\cdot 2^{n-2}+2}.
\end{equation}
\end{prop}

\begin{cor}
\label{prop:csirik-corollary}
Let $n\ge 5$.  The number $c_n$ of optimal strategies for Player I
in the Penney-Ante game with strings of length $n$ is given by 
$c_n=2c_{n-1}^*$, where $c_m^*$ denotes the number of strings of length
$m$ beginning with $HT$ and ending in $TH$
that have autocorrelation $2^{m-1}+1$. 
\end{cor}

\begin{proof}
By symmetry there are an equal number of optimal strings of each of the
two forms described in Proposition \ref{prop:csirik-theorem3}.
Therefore the number $c_n$ of optimal strings is twice
the number of such strings of the first form, i.e., 
$HTa_3\dots a_{n-3}THH$, and 
those strings are in one-to-one correspondence with the strings
of length $n-1$ counted by $c_{n-1}^*$. Hence $c_n=2c_{n-1}^*$. 
\end{proof}

In light of Corollary \ref{prop:csirik-corollary},
the desired recurrence
\eqref{eq:cn-recurrence-DUP} for $c_n$ can be restated as a recurrence for
the numbers $c_n^*$:
\begin{equation}
    \label{eq:cnstar-recurrence}
        c_n^*=2c_{n-1}^*+(-1)^nc_{\fl{(n+1)/2}}^*\quad (n\ge 5).
\end{equation}
Considering separately the case of even and odd values of $n$, we can
rewrite \eqref{eq:cnstar-recurrence} as the pair of recurrences
\begin{align}
    \label{eq:cnstar-recurrence-odd}
    c_{2m+1}^*&=2c_{2m}^*-c_{m+1}^*\quad (m\ge 2),
    \\
    \label{eq:cnstar-recurrence-even}
    c_{2m}^*&=2c_{2m-1}^*+ c_{m}^*\quad (m\ge 3).
\end{align}
To prove Theorem \ref{thm:cn-recurrence}, it suffices to establish 
the relations \eqref{eq:cnstar-recurrence-odd} and
\eqref{eq:cnstar-recurrence-even}.

\begin{proof}[Proof of \eqref{eq:cnstar-recurrence-odd}]
We will prove \eqref{eq:cnstar-recurrence-odd} by showing that, for
$m\ge2$,
\begin{equation}
\label{eq:cnstar-recurrence-odd2}
2c_{2m}^*=c_{2m+1}^*+c_{m+1}^*.
\end{equation}

Consider a string $A$ counted by $c_{2m}^*$, i.e., 
a string of length $2m$ of the form
\begin{equation}
\label{eq:A-string}
A=HTa_3\dots a_{2m-2}TH 
\end{equation}
with autocorrelation $100\dots 01$.
Write $A=A_1A_2$, where $A_1$ is the string consisting of the first $m$
bits of $A$, and $A_2$ is the string consisting of the second $m$ bits of
$A$, i.e., 
\begin{equation}
\label{eq:A2-A1-string}
A_1=HTa_3\dots a_{m},\quad
A_2=a_{m+1}\dots a_{2m-2}TH.
\end{equation}
Given $X\in\{H,T\}$, define a string $A^X$  of length $2m+1$ by
\begin{align}
\label{eq:AX-string}
A^X=A_1XA_2=HTa_3\dots a_{m}Xa_{m+1}\dots a_{2m-2}TH.
\end{align}

Now note that when calculating the autocorrelation of each of the 
three strings
$A$ and $A^X$, $X\in\{H,T\}$, 
the last $m$ bits are based on comparing $A_2$ with
$A_1$ and thus are the same for each of these three strings (cf. the
proof of Lemma \ref{lem:autocorr-sum}).
Since, by assumption, the string $A$ has autocorrelation
$100\dots 01$, and hence ends in the $m$-bit string $00\dots 01$, 
the autocorrelations of the strings $A^X$ must   
end in the same $m$-bit string  $00\dots 01$. By Lemma
\ref{lem:correctAutocorrelation}(ii) this is only possible if $A^X$
has an autocorrelation of one of the following two forms:
\begin{equation}
\label{eq:AX-autocorrelation}
\mbox{(I)}\quad 1\underbrace{00\dots 0}_{2m-1}1\quad\text{ or }\quad 
\mbox{(II)}\quad  1\underbrace{00\dots 0}_{m-1}1\underbrace{0\dots 0}_{m-1}1.
\end{equation}
Conversely, any string $A^X$ of the form \eqref{eq:AX-string}
with autocorrelation
\eqref{eq:AX-autocorrelation} corresponds to a string $A$
of the form \eqref{eq:A-string} with autocorrelation  ending in the
$m$-bit string $00\dots01$. By Lemma 
\ref{lem:correctAutocorrelation}(i) each such string $A$ has
autocorrelation $100\dots 01$ and thus is counted by $c_{2m}^*$.
Since each string $A$ counted by $c_{2m}^*$
gives rise to two strings $A^X$ 
with autocorrelation \eqref{eq:AX-autocorrelation}
(one for each choice of $X$), the total number of strings 
$A^X$ with autocorrelation \eqref{eq:AX-autocorrelation} must be 
$2c_{2m}^*$.

On the other hand, we can also count the number of such strings $A^X$  
by counting separately those whose autocorrelation is given by (I) in
\eqref{eq:AX-autocorrelation}
and those whose autocorrelation is given by
(II)  in \eqref{eq:AX-autocorrelation}.
The strings $A^X$ with autocorrelation (I) are exactly those
counted by $c_{2m+1}^*$, so the number of such strings is $c_{2m+1}^*$. 

The strings $A^X$ with autocorrelation (II)  
can be counted as follows: Observe that the last
$m+1$ bits of the autocorrelation of $A=A_1XA_2$ are based on the
comparison of the strings $XA_2=Xa_{m+1}\dots a_{2m-2}TH$
and $A_1X=HTa_3\dots a_mX$, and thus can only be of the
form $100\dots 01$ if these two strings are equal to a common string 
$B=HTb_3\dots b_{m-1}TH$ of length $m+1$ with autocorrelation $100\dots
01$, i.e., a string counted by $c_{m+1}^*$. 
Conversely, any such string $B$ corresponds to a string $A^X$
with autocorrelation (II).  Thus, the number of strings $A^X$
with autocorrelation (II) is exactly $c_{m+1}^*$.

It follows that $2c_{2m}^*=c_{2m+1}^*+c_{m+1}^*$, which proves the desired
relation \eqref{eq:cnstar-recurrence-odd2}.
\end{proof}

\begin{proof}[Proof of \eqref{eq:cnstar-recurrence-even}]
We will show that
\begin{equation}
\label{eq:cnstar-recurrence-even2}
4c_{2m}^*=c_{2m+2}^*+c_{m+1}^*.
\end{equation}
Substituting the relation \eqref{eq:cnstar-recurrence-odd2}
into \eqref{eq:cnstar-recurrence-even2},
we obtain $2c_{2m+1}^*=c_{2m+2}^* -c_{m+1}^*$, which yields 
the desired relation \eqref{eq:cnstar-recurrence-even} after shifting 
the index.

To prove \eqref{eq:cnstar-recurrence-even2}, we begin as before by
letting $A$ be a string counted by $c_{2m}^*$, i.e.,  
a string of length $2m$ of the form \eqref{eq:A-string},
with autocorrelation $100\dots 01$. We define 
$A_1$ and $A_2$ by \eqref{eq:A2-A1-string},
and consider the four strings $A^{XY}$ of length $2m+2$ obtained by inserting 
a two-bit string $XY$ (with $X,Y\in \{H,T\}$) between $A_1$ and $A_2$; that is, 
\begin{equation}
\label{eq:AXY-string}
A^{XY}=A_1XYA_2.
\end{equation}
Arguing as before, we see that the last $m$ bits of the autocorrelation of each
such string $A^{XY}$ are equal to the last $m$ bits of the autocorrelation
of the string $A$ and hence must be $00\dots 01$.  By Lemma
\ref{lem:correctAutocorrelation}(iii) it follows that $A^{XY}$ must have  
autocorrelation of the form    
\begin{equation}
\label{eq:AXY-autocorrelations}
\mbox{(I)'}\quad 1\underbrace{00\dots 0}_{2m}1\quad\text{ or }\quad 
\mbox{(II)'}\quad 1\underbrace{00\dots 0}_{m}1\underbrace{0\dots 0}_{m-1}1.
\end{equation}
The number of strings $A^{XY}$ with autocorrelation (I)' is
exactly $c_{2m+2}^*$. As before, we see that the case  of 
autocorrelation (II)' occurs if and only if the strings $XA_1$ and $A_2Y$
are equal to a common string of length $m+1$ of the form $B=HTb_3\dots
b_{m-1}TH$ with autocorrelation $100\dots 01$.
Since there are exactly $c_{m+1}^*$ such strings $B$, 
the number of strings $A^{XY}$ with autocorrelation (II)' is also $c_{m+1}^*$.

Since there are $c_{2m}^*$ strings $A$, and each of these strings
corresponds to exactly four strings $A^{XY}$, we obtain 
$4c_{2m}^*=c_{2m+2}^*+ c_{m+1}^*$. This is the desired relation
\eqref{eq:cnstar-recurrence-even2}. 
\end{proof}

\section{Proof of Theorem \protect\ref{thm:cn-asymptotic}}
\label{sec:theorem2-proof}

Theorem \ref{thm:cn-asymptotic} states that the number $c_n$ satisfies 
the asymptotic relation \eqref{eq:cn-asymptotic}, i.e., 
\begin{equation}
\label{eq:cn-asymptotic-DUP}
                c_n=
                \begin{cases}
                \alpha 2^n + O\left(2^{n/4}\right) 
                & \text{if $n$ is even,}
                \\
                \alpha \left(2^n + 2^{\fl{n/2}+1}\right)+ O(2^{n/4})
                & \text{if $n$ is odd,}
\end{cases}
        \end{equation}
where $\alpha=0.040602\dots$ is a numerical constant.

To prove \eqref{eq:cn-asymptotic-DUP}, we will employ an
iterative procedure based on the recurrence  \eqref{eq:cn-recurrence} 
of Theorem \ref{thm:cn-recurrence}.

We first rewrite  \eqref{eq:cn-recurrence} as the pair of recurrences 
\begin{align}
\label{eq:cn-recurrence-odd}
c_{2m+1}&=2c_{2m}+c_{m+1}\quad (m\ge 3),
\\
\label{eq:cn-recurrence-even}
c_{2m}&=2c_{2m-1}-c_{m+1}\quad (m\ge 3).
\end{align}
Iterating \eqref{eq:cn-recurrence-odd}
and \eqref{eq:cn-recurrence-even} once yields
\begin{align}
\label{eq:cn-recurrence-odd2}
c_{2m+1}&=
2(2c_{2m-1}-c_{m+1})+c_{m+1}=
4c_{2m-1}-c_{m+1}\quad (m\ge 4),
\\
\label{eq:cn-recurrence-even2}
c_{2m}&=
 2(2c_{2m-2}+c_m)-c_{m+1}=
4c_{2m-2}+2c_m-c_{m+1}\quad (m\ge 4).
\end{align}

To bootstrap our iterative argument, we need a relatively crude initial
bound for $c_n$.  The following lemma provides such a bound.
\begin{lem}
\label{lem:cn-initial-bounds}
We have
\begin{equation}
\label{eq:cn-initial-bounds}
2^{n-6}\le c_n\le 2^{n-4}\quad (n\ge 5).
\end{equation}
\end{lem}

\begin{proof}
For $5\le n\le 8$ the bounds \eqref{eq:cn-initial-bounds} can be
verified directly using Table  \ref{table:cn-initial-values}. Thus it
suffices to prove these bounds for $n\ge9$.

For the upper bound in \eqref{eq:cn-initial-bounds}, note that
\eqref{eq:cn-recurrence-odd2} implies $c_{2m+1}\le 4c_{2m-1}$ for all
$m\ge4$. Iterating this inequality $m-3$ times yields
\[
c_{2m+1}\le
4^{m-3}c_7=2^{2m-6}\cdot 6<2^{2m+1-4} \quad (m\ge 4),
\]
which is the desired upper bound for odd values $n\ge 9$. The bound for even
values $n$ then follows on noting that, by \eqref{eq:cn-recurrence-even},
$c_{2m}\le 2 c_{2m-1}\le 2\cdot 2^{2m-1-4}= 2^{2m-4}$ for $m\ge3$.

We now turn to the lower bound in \eqref{eq:cn-initial-bounds}. Using
\eqref{eq:cn-recurrence-odd2} along with the upper bound $c_{m+1}\le
2^{m+1-4}$  we obtain
\begin{equation*}
c_{2m+1}=4c_{2m-1}-c_{m+1}\ge 4c_{2m-1}-2^{m-3}\quad (m\ge 4).
\end{equation*}
Iterating this inequality $m-3$ times gives
\begin{align*}
c_{2m+1}&\ge 4^{m-3}c_{7}-S=2^{2m-6}\cdot 6 -S,
\end{align*}
where
\begin{align*}
S&=\sum_{i=0}^{m-4} 2^{m-3-i}4^{i}
=2^{m-3}\sum_{i=0}^{m-4} 2^i=  2^{m-3}(2^{m-3}-1)< 2^{2m-6}.
\end{align*}
Hence
\begin{equation}
\label{eq:cn-lower-bound-odd}
c_{2m+1}\ge 2^{2m-6}\cdot 6-2^{2m-6}> 2^{2m+1-6} \quad (m\ge 4).
\end{equation}
An analogous argument, based on
\eqref{eq:cn-recurrence-even2}, yields
\begin{align}
\label{eq:cn-lower-bound-even}
c_{2m}&\ge 4c_{2m-2} -c_{m+1}
\\
\notag
&\ge 4c_{2m-2}-2^{m-3}
\\
\notag
&\ge 4^{m-4}c_8-2^{2m-6}\\
\notag
&= 2^{2m-8}\cdot 10 - 2^{2m-6} > 2^{2m-6}
\quad (m\ge 4).
\end{align}
The desired lower bound, $c_n\ge 2^{n-6}$, follows (for $n\ge9$) from
\eqref{eq:cn-lower-bound-odd} and \eqref{eq:cn-lower-bound-even}.
This completes the proof of Lemma \ref{lem:cn-initial-bounds}.
\end{proof}

Next, we rescale $c_n$ by setting
\begin{equation}
\label{eq:dn-definition}
d_n=2^{-n}c_n.
\end{equation}
The inequalities \eqref{eq:cn-initial-bounds}
of Lemma \ref{lem:cn-initial-bounds} imply
\begin{equation}
\label{eq:dn-initial-bounds}
\frac{1}{64}\le d_n\le \frac{1}{16}\quad (n\ge 5),
\end{equation}
so the sequence $\{d_n\}$ is bounded above and below by positive
constants.
In the following lemma, we show that this sequence converges.

\begin{lem}
\label{lem:dn-convergence1}
The limit
\begin{equation}
\label{eq:dn-limit}
\alpha=\lim_{n\to\infty}d_n
\end{equation}
exists and is strictly positive. Moreover, as $n\to\infty$, we have
\begin{equation}
\label{eq:dn-asymptotic1}
d_n=\alpha + O\left(2^{-n/2}\right).
\end{equation}
\end{lem}

\begin{proof}
Substituting $c_n=2^nd_n$ into the recurrences
\eqref{eq:cn-recurrence-odd}
and \eqref{eq:cn-recurrence-even}, we obtain
\begin{align}
\label{eq:dn-recurrence-odd}
d_{2m+1}&=d_{2m}+2^{-m}d_{m+1}\quad (m\ge 3),
\\
\label{eq:dn-recurrence-even}
d_{2m}&=d_{2m-1}-2^{-m+1}d_{m+1}\quad (m\ge 3).
\end{align}
Since, by \eqref{eq:dn-initial-bounds}, $d_n$ is bounded,
the second term on the right of
\eqref{eq:dn-recurrence-odd}
and \eqref{eq:dn-recurrence-even} is of order $O(2^{-m})$, so we have
\begin{equation*}
d_n=d_{n-1}+O\left(2^{-n/2}\right)\quad (n\ge 6).
\end{equation*}
Iterating this relation gives, for any integer $k\ge1$,
\begin{equation}
\label{eq:dn-estimate1}
d_{n}=d_{n+k}+
O\left(\sum_{i=1}^k 2^{-(n+i)/2}\right)
=d_{n+k}+ O\left(2^{-n/2}\right)
\quad (n\ge 5),
\end{equation}
where the constant implied by the $O$-notation is independent of $k$ and $n$.
Hence the sequence $\{d_n\}$ is a Cauchy sequence and therefore
has a limit, $\alpha=\lim_{n\to\infty} d_n$.

It follows from \eqref{eq:dn-initial-bounds} that $\alpha$ is strictly
positive.
Moreover, letting $k\to\infty$ in \eqref{eq:dn-estimate1}, we obtain
$d_n=\alpha+O(2^{-n/2})$, which is the desired estimate
\eqref{eq:dn-asymptotic1}. This completes the proof of Lemma
\ref{lem:dn-convergence1}.  \end{proof}

\begin{lem}
\label{lem:dn-asymptotic2}
We have
\begin{align}
\label{eq:dn-asymptotic-even}
d_{2m}&=\alpha+O\left(2^{-(3/2)m}\right)
\quad (m\ge 4),
\\
\label{eq:dn-asymptotic-odd}
d_{2m+1}&=\alpha\left(1+2^{-m}\right)+ O\left(2^{-(3/2)m}\right)
\quad (m\ge 4).
\end{align}
\end{lem}

\begin{proof}
Iterating \eqref{eq:dn-recurrence-odd}
and \eqref{eq:dn-recurrence-even} yields
\begin{align}
\label{eq:dn-recurrence-even2}
d_{2m}&=d_{2m-2} + 2^{-m+1}\left(d_m-d_{m+1}\right)\quad (m\ge 4).
\end{align}
Since, by Lemma \ref{lem:dn-convergence1},
$d_m=\alpha +O(2^{-m/2})$ and $d_{m+1}=\alpha +O(2^{-m/2})$,
the last term in \eqref{eq:dn-recurrence-even2} is of order
$O(2^{-m+1}\cdot 2^{-m/2})=O(2^{-(3/2)m})$, so we have
\[
d_{2m} = d_{2m-2}
+ O\left(2^{-(3/2)m}\right) \quad (m\ge 4).
\]
It follows that, for any $k\ge 1$,
\begin{align*}
d_{2m}&=d_{2m+2k}+
O\left(\sum_{i=1}^k2^{-(3/2)(m+i)}\right) =d_{2m+2k}+
O\left(2^{-(3/2)m}\right).
\end{align*}
Letting $k\to\infty$, we obtain the first estimate of the lemma,
\eqref{eq:dn-asymptotic-even}.

The second estimate, \eqref{eq:dn-asymptotic-odd}, follows on noting
that, by \eqref{eq:dn-recurrence-odd} and
\eqref{eq:dn-asymptotic1},
\begin{align*}
d_{2m+1}&=d_{2m}+2^{-m}d_{m+1}
\\
&=\alpha +O\left(2^{-(3/2)m}\right)
+2^{-m}\left(\alpha+O\left(2^{-(m+1)/2}\right)\right)
\\
&=\alpha \left(1+2^{-m}\right)
+O\left(2^{-(3/2)m}\right)
\end{align*}
This completes the proof of Lemma \ref{lem:dn-asymptotic2}.
\end{proof}

\begin{proof}[Proof of Theorem \ref{thm:cn-asymptotic}]
The desired asymptotic estimate \eqref{eq:cn-asymptotic-DUP}
follows from the estimate
\eqref{eq:dn-asymptotic-even} of Lemma
\ref{lem:dn-asymptotic2} when $n=2m$ is even,
and from \eqref{eq:dn-asymptotic-odd} when $n=2m+1$ is odd.
\end{proof}

It is clear that the iterative procedure we have used in this proof
could, in principle, be continued to extract further main terms from the
error term $O(2^{n/4})$ in \eqref{eq:cn-asymptotic}.  For example, one
additional iteration would yield an additional
main term of size $2^{n/4}$, with a
coefficient depending on the remainder of $n$ modulo $4$, along with an
error term of the form $O(2^{n/8})$.

\section{Proof of Theorem \protect\ref{thm:alpha-series}}
\label{sec:theorem3-proof}
	
Thereom \ref{thm:alpha-series} states that the 
constant $\alpha$ in Theorem \ref{thm:cn-asymptotic}
satisfies  \eqref{eq:alpha-series}, i.e., 
	\begin{equation}
        \label{eq:alpha-series-DUP}
        \alpha = \frac{1}{16}-2\sum_{n=4}^\infty \frac{c_n}{4^{n}}.
    \end{equation}
Our proof of \eqref{eq:alpha-series-DUP} is based on the following
lemma.

\begin{lem}
    \label{lem:alpha-cn-sum}
    We have
    \begin{align}
    \label{eq:cn-finite-sum}
        c_{2m+1}=4^{m-2}c_5-\sum_{i=4}^{m+1}c_i4^{m+1-i}
	\quad (m\ge 3).
    \end{align}
\end{lem}

\begin{proof} 
We proceed by induction.  For $m=3$, \eqref{eq:cn-finite-sum} reduces to
$c_7=4\cdot c_5-c_4$, which can be verified directly using the values
$c_7=6$ and $c_5=c_4=2$ from Table \ref{table:cn-initial-values}.

Now let $m\ge 3$ and assume \eqref{eq:cn-finite-sum} holds for $m$.
Then, using the recurrence \eqref{eq:cn-recurrence-odd2}, 
we have
\begin{align*}
    c_{2(m+1)+1}&=4c_{2m+1}-c_{(m+1)+1}
   \\ 
    &= 4\left(4^{m-2}c_5-\sum_{i=4}^{m+1}c_i4^{m+1-i}\right)-c_{m+2}
    \\
 &=4^{(m+1)-2}c_5-\sum_{i=4}^{(m+1)+1}c_i4^{(m+1)+1-i}.
\end{align*}
  Hence \eqref{eq:cn-finite-sum} holds with $m+1$ in place of $m$,
  completing the induction.  
\end{proof}

\begin{proof}[Proof of Theorem \ref{thm:alpha-series}]
Dividing both sides of \eqref{eq:cn-finite-sum} by $2^{2m+1}$ we obtain 
\begin{align}
\label{eq:cn-finite-sum2}
\frac{c_{2m+1}}{2^{2m+1}} &= \frac{c_5}{2^{5}}
-2\sum_{i=4}^{m+1}\frac{c_i}{4^i}
=\frac{1}{16}-2\sum_{i=1}^{m+1}\frac{c_i}{4^i},
\end{align}
upon substituting the value $c_5=2$.
Letting $m\to\infty$ in \eqref{eq:cn-finite-sum2} and noting that, by 
Theorem \ref{thm:cn-asymptotic}, $\lim_{n\to\infty}c_n2^{-n}=\alpha$, 
yields the desired formula \eqref{eq:alpha-series-DUP} for $\alpha$. 
\end{proof}

\section{Proof of Theorem \protect\ref{thm:flipped-optimal-moves}}
\label{sec:theorem4-proof}

Theorem \ref{thm:flipped-optimal-moves} asserts 
that $HH\dots H$ and $TT\dots T$ are the unique
optimal strings for Player I in the flipped Penney-Ante game, and that
with these strings Player I has even odds, i.e., a winning probability
of $1/2$, under optimal play by Player II.

Recall that in the flipped game the player
whose string appears \emph{last} in a random head/tail sequence wins
the game. Thus, if $A$ and $B$ are the strings chosen by Players I and
II, respectively, then the odds in favor of Player I are 
\begin{equation}
\label{eq:flipped-game-odds}
q(A,B)= \frac{P(\text{$B$ appears before $A$})}
{P(\text{$A$ appears before $B$})},
\end{equation}
which, by Conway's formula \eqref{eq:conway-algorithm}, can be expressed 
in terms of Conway numbers:
\begin{equation}
\label{eq:qAB-definition}
q(A,B)
=\frac{C(A,A)-C(A,B)}{C(B,B)-C(B,A)}.
\end{equation}
To prove Theorem \ref{thm:flipped-optimal-moves}, we need to show 
that the strings $A=HH\dots H$ and $A=TT\dots T$
are the unique strings for which 
$q(A,B)\ge1$ for all choices of $B\not=A$, and that  
equality holds for at least one such choice. 
This will follow from Lemma \ref{lem:flipped-game-odds} below.
Here, and in the remainder of this section, all strings are
assumed to be of a fixed length $n\ge3$.

\pagebreak[3]

\begin{lem}
\label{lem:flipped-game-odds}
\mbox{}
\begin{itemize}
\item[(i)] If $A=HH\dots H$,
then for any string $B\not=A$ we have $q(A,B)\ge 1$, with equality 
holding if and only if $B$ is one of the following 
two $n$-bit strings:
\begin{equation}
\label{eq:HHH-best-response}
TT\dots T, \quad HH\dots HT.
\end{equation}

\item[(ii)]
If $A=TT\dots T$,
then for any string $B\not=A$ we have $q(A,B)\ge 1$, with equality 
holding if and only if $B$ is one of the following 
two $n$-bit strings:
\begin{equation}
\label{eq:TTT-best-response}
HH\dots H, \quad TT\dots TH. 
\end{equation}

\item[(iii)] If $A$ is not of the form 
$A=HH\dots H$ or $A=TT\dots T$,
then there exists a string $B\not=A$ such that $q(A,B)<1$.

\end{itemize}
\end{lem}

\begin{proof}
(i)
Assume that $A$ is the $n$-bit string $HH\dots H$ and $B=b_1\dots b_n$
is  a string of length $n$ different from $A$.

Let $s$ be the number of leading bits $H$ in $B$, and let $t$ be the
number of trailing bits $H$ in $B$. Since the string $B$ is different
from the string $A=HH\dots H$, it must contain at least one $T$, so we have
$0\le s,t\le n-1$ and  $s+t<n$. 

Since for each $i\in\{1,\dots,n\}$, the prefix and suffix of length $i$
of $A=HH\dots H$ match, all bits in the Conway number $C(A,A)$ are $1$
and we thus have 
\begin{equation}
\label{eq:CAA}
C(A,A)=\sum_{i=0}^{n-1}2^i=2^n-1.
\end{equation}

Next, note that
at each step in the computation of the Conway number
$C(A,B)=C(HH\dots H,B)$, a 
prefix of $B$ is compared with a suffix of the string $HH\dots H$
of the same length,  so a match occurs if and only if the
prefix consists of all $H$'s. This happens for the last $s$ comparisons,  
so the final $s$ bits in the Conway number $C(A,B)$ are equal to $1$, while all
other bits are $0$. Hence we have 
\begin{equation}
\label{eq:CAB}
C(A,B)=\sum_{i=0}^{s-1}2^i=2^s-1.
\end{equation}
An analogous argument yields 
\begin{equation}
\label{eq:CBA}
C(B,A)=\sum_{i=0}^{t-1}2^i=2^t-1.
\end{equation}

Finally consider the Conway number $C(B,B)$.  
Since $B$ matches itself, the first bit in this number must
be $1$. If the second bit of $C(B,B)$ is
also $1$, then we must have $b_i=b_{i+1}$ for $i=1,2,\dots,n-1$
and hence $b_1=b_2=\dots =b_{n}$.  Since we assumed that $B$ 
is different from the string $A=HH\dots H$, $B$ must be equal to the
string $TT\dots T$.  It follows that  $s=t=0$ and  
therefore, by \eqref{eq:CAB} and \eqref{eq:CBA}, $C(A,B)=C(B,A)=0$. 
Moreover, using the same argument as for 
\eqref{eq:CAA} we see that $C(B,B)=C(TT\dots T, TT\dots T)=2^n-1$.
Hence we have
\begin{equation}
\label{eq:qATTTT}
q(A,TT\dots T)=\frac{(2^{n}-1)-0}{(2^{n}-1)-0}=1.
\end{equation}

If the second bit of $C(B,B)$ is $0$, then
\begin{equation}
\label{eq:CBB-bound}
C(B,B)\le 2^{n-1}+2^{n-3}+\dots +2^0=2^{n}-1-2^{n-2}.
\end{equation}
Substituting \eqref{eq:CAA}, \eqref{eq:CAB}, \eqref{eq:CBA}, 
and \eqref{eq:CBB-bound} into \eqref{eq:qAB-definition}, we obtain the bound
\begin{align}
\label{eq:qAB-bound}
q(A,B)\ge 
\frac{(2^{n}-1)-(2^s-1)}{(2^n-1-2^{n-2})-(2^t-1)}
\ge  \frac{2^n-2^s}{2^n-1-2^{n-2}}.
\end{align}
It follows that
$q(A,B)>1$ unless $2^s\ge 2^{n-2}+1$.
The latter case can only occur if $s=n-1$ and $t=0$, i.e., 
if $B$ is the string $B=HH\dots HT$. By 
\eqref{eq:CAB} and \eqref{eq:CBA} we have in this case
$C(A,B)=2^s-1=2^{n-1}-1$ and $C(B,A)=2^t-1=0$. Moreover, in the computation
of the autocorrelation of $B=HH\dots HT$, a match occurs only at the first
bit, so we have $C(B,B)=2^{n-1}$. We thus obtain
\begin{align}
\label{eq:qAHHHT}
q(A,HH\dots HT)=\frac{(2^{n}-1)-(2^{n-1}-1)} {2^{n-1}-0} =1.
\end{align}
Altogether we have shown that $q(A,B)>1$ if $B$ is \emph{not} of the
form $TT\dots T$ or $HH\dots HT$, and $q(A,B)=1$ if $B$ is of this form.
This proves part (i) of the lemma.

\bigskip

(ii) This part follows by interchanging the roles of $H$ and $T$ in the proof
of part (i).

\bigskip

(iii) 
Suppose $A$ is not of the form $HH\dots H$ or $TT\dots T$.
Let $B_1=HH\dots H$ and $B_2=TT\dots T$. 
We will show that $q(A,B)<1$ holds for at least one of the strings 
$B=B_1$ and $B=B_2$.

Applying part (i) with 
$A$ replaced by $B_1$, we obtain $q(B_1,A)>1$ if $A$ is not of the form
(I) $HH\dots HT$ (note that, by our assumption, $A$ is not of the form $HH
\dots H$ or $TT\dots T$). Similarly, applying part (ii)
we obtain $q(B_2,A)>1$ if $A$ is not of the form (II) $TT\dots TH$.
But since a string cannot be 
equal to both of the strings (I) and (II), it follows that
at least one of the inequalities $q(B_1,A)>1$ and $q(B_2,A)>1$ holds.
Since $q(A,B)=1/q(B,A)$, we conclude that
at least one of the inequalities  $q(A,B_1)<1$ and $q(A,B_2)<1$ holds.
This proves part (iii) of the lemma and completes the proof of Theorem
\ref{thm:flipped-optimal-moves}.
\end{proof}

\section{Open Problems and Conjectures}
\label{sec:further-results}

In this section we discuss some open problems related to our results, 
present some numerical data, and formulate several conjectures suggested
by the data.

\paragraph{Arithmetic nature of $\alpha$.}
Expanding the proportionality constant $\alpha$ in
Theorem \ref{thm:cn-asymptotic} in base $2$ gives
\begin{equation}
\label{eq:alpha-binary}
\alpha=0.001010011001100111010000101011000001011010010011010100101\dots
\end{equation}
There is no obvious periodicity pattern in this expansion, so it seems  likely that $\alpha$ is irrational. In fact, numerical data  based on the first $160,000$ bits in this expansion  suggests that $\alpha$ is a \emph{normal} number with respect to base $2$, i.e., that each  binary string of length $n$ occurs with the expected frequency, $1/2^n$, in the sequence of digits of $\alpha$. Our computations indicate that this is indeed the case for strings  of length $n\le 8$.

The integer sequence that encodes the positions of the $1$-bits in the expansion \eqref{eq:alpha-binary}
is $3$, $5$, $8$, $9$, $12$, $13$, $16$, $17$, $18$, $20$,$\dots $. 
This sequence does not seem to have a closed form, and it is not listed in
the \emph{On-Line Encyclopedia of Integer Sequences} \cite{oeis}.

\paragraph{Winning probabilities under random instead of optimal strategies.}
Our basic assumption in this paper---as in prior work such as 
Guibas-Odlyzko \cite{guibas-odlyzko1981}, Csirik \cite{csirik1992}, and 
Felix \cite{felix2006}---was that both players were skilled players,
with each employing a strategy that maximizes their respective winning
probabilities.

It is natural to ask how much of a penalty a player
incurs by using instead a random strategy, i.e., by choosing a string at
random from all $2^n$ strings of length $n$.  Such a random strategy could model an unskilled player who is not familiar with the theory of the Penney-Ante game.

To investigate this question, let $\psIsII$ be the probability that
Player II wins in the Penney-Ante game on strings of length $n$ assuming 
Player I employs strategy $\sI$ and Player II employs strategy $\sII$.
We restrict to the case when $\sI,\sII\in\{\opt,\rand\}$, where $\opt$
denotes a strategy that is optimal (in the sense of maximizing the player's winning probability assuming optimal play by the opponent), while $\rand$ denotes the strategy
in which the player chooses one of the $2^n$ strings at random. 

In particular, $\poptopt$  is the probability that Player II wins assuming both
players play optimally; by Csirik's result (Proposition
\ref{prop:csirik-theorem3}), this probability is equal to 
\begin{equation}
\label{eq:poptopt-formula}
\poptopt=\frac{2^{n-1}+1}{3\cdot 2^{n-2}+2}=\frac{2}{3}-\frac1{3(3\cdot
2^{n-2}+2)}.
\end{equation}
We are interested in comparing this probability to the probabilities
$\prandopt$ and $\poptrand$, the winning probabilities for Player II assuming Player I (resp.~Player II) 
plays randomly while Player II (resp.~Player I) maintains an   
optimal strategy.  How much of a reduction in the winning probability does a player incur by using a random strategy instead of an optimal strategy?  

We first consider the case when Player I plays randomly, while Player II maintains an
optimal strategy. 
Table \ref{table:random-I-table} shows the
winning probabilities for Player II assuming either optimal play by Player I (column $\poptopt$) or random play by Player I (column $\prandopt$). The probabilities $\poptopt$ here are those given by the exact formula  \eqref{eq:poptopt-formula}, 
while the probabilities $\prandopt$ were determined experimentally, using computer simulations. As expected, under random play by Player I, Player II has an increased winning probability, but the difference appears to be exponentially small:  For example, for $n=15$ the two probabilities 
agree in their first three digits, while for $n=19$ they agree in their first four digits and for $n=22$ they agree in their first five digits.
\begin{table}[H]
\begin{center}
\renewcommand{\arraystretch}{1.1}
\begin{tabular}{|c|c|c|c|c|}
\hline
$n$ & $\poptopt$ & $\prandopt$ & $2^n(\prandopt-2/3)/n$ \\
\hline
5 & 0.65384615 & 0.71868171 & 0.33289627 \\
6 & 0.66000000 & 0.69865016 & 0.34115722 \\
7 & 0.66326531 & 0.68739336 & 0.37900236 \\
8 & 0.66494845 & 0.67913922 & 0.39912157 \\
9 & 0.66580311 & 0.67411092 & 0.42349539 \\
10 & 0.66623377 & 0.67094023 & 0.43761240 \\
11 & 0.66644993 & 0.66910562 & 0.45408969 \\
12 & 0.66655823 & 0.66803837 & 0.46820972 \\
13 & 0.66661243 & 0.66743344 & 0.48318813 \\
14 & 0.66663954 & 0.66708843 & 0.49358630 \\
15 & 0.66665310 & 0.66689731 & 0.50385310 \\
16 & 0.66665989 & 0.66679196 & 0.51318555 \\
17 & 0.66666328 & 0.66673437 & 0.52202351 \\
18 & 0.66666497 & 0.66670302 & 0.52947573 \\
19 & 0.66666582 & 0.66668611 & 0.53649966 \\
20 & 0.66666624 & 0.66667702 & 0.54277612 \\
21 & 0.66666645 & 0.66667216 & 0.54859471 \\
22 & 0.66666656 & 0.66666957 & 0.55383968 \\
23 & 0.66666661 & 0.66666820 & 0.55868718 \\
24 & 0.66666664 & 0.66666747 & 0.56313417 \\
\hline
\end{tabular}
\caption{Optimal versus random play by Player I.}

\label{table:random-I-table}
\end{center}
\end{table}

The last column of Table \ref{table:random-I-table} suggests 
the following more precise conjecture for the behavior 
of $\prandopt$ as $n\to\infty$.

\begin{conj}
\label{conj:rand-opt}
The probability $\prandopt$ that Player II wins assuming random play by
Player I and optimal play by Player II satisfies
\begin{equation}
\label{eq:conj-rand-opt}
\prandopt=\frac{2}{3}+O\left(\frac {n}{2^{n}}\right)\quad (n\to\infty).
\end{equation}
\end{conj}
In fact, Figure 
\ref{fig:random-I-figure} below suggests that
$2^n(\prandopt-2/3)$ is asymptotically linear.   If so, 
the asymptotic estimate \eqref{eq:conj-rand-opt} could be strengthened to 
\begin{equation}
\label{eq:conj-rand-opt-strong}
\prandopt-\frac{2}{3}\sim\frac {c n}{2^{n}}\quad (n\to\infty),
\end{equation}
where $c$ is a positive constant.

\begin{figure}[H]
    \centering
    \includegraphics[width=0.9\textwidth]{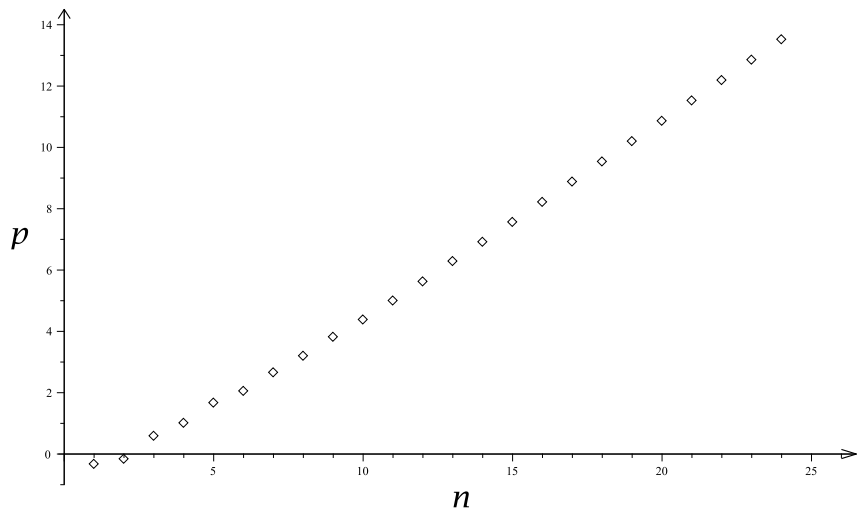}
    \caption{Plot of $2^n(\prandopt-2/3)$.}
    \label{fig:random-I-figure}
\end{figure}

We can similarly ask how much of a reduction in winning probabilities 
Player II incurs when employing a random strategy instead of playing
optimally.  Table \ref{table:random-II-table1} shows the
winning probabilities for Player II assuming optimal play by Player I and 
either optimal or random  play by Player II. 
As can be seen from this table, the difference between an optimal strategy and a random
strategy is far more dramatic for Player II than it is for Player I: Under
random play, Player II's winning probabilities decrease from just below
$2/3$ to just below $1/2$. 

\begin{table}[H]
\begin{center}
\renewcommand{\arraystretch}{1.1}
\begin{tabular}{|c|c|c|c|}
\hline
$n$ & $\poptopt$ & $\poptrand$ & $2^n(1/2-\poptrand)/n$\\
\hline
5 & 0.65384615 & 0.46497915 & 0.22413343 \\
6 & 0.66000000 & 0.47844501 & 0.22991993 \\
7 & 0.66326531 & 0.48728813 & 0.23244566 \\
8 & 0.66494845 & 0.49267595 & 0.23436966 \\
9 & 0.66580311 & 0.49585625 & 0.23573334 \\
10 & 0.66623377 & 0.49768613 & 0.23694059 \\
11 & 0.66644993 & 0.49872187 & 0.23796489 \\
12 & 0.66655823 & 0.49930014 & 0.23888612 \\
13 & 0.66661243 & 0.49961965 & 0.23967636 \\
14 & 0.66663954 & 0.49979460 & 0.24038122 \\
15 & 0.66665310 & 0.49988968 & 0.24100430 \\
16 & 0.66665989 & 0.49994103 & 0.24155677 \\
17 & 0.66666328 & 0.49996861 & 0.24204743 \\
18 & 0.66666497 & 0.49998335 & 0.24248627 \\
19 & 0.66666582 & 0.49999120 & 0.24287996 \\
20 & 0.66666624 & 0.49999536 & 0.24323508 \\
21 & 0.66666645 & 0.49999756 & 0.24355672 \\
22 & 0.66666656 & 0.49999872 & 0.24384934 \\
23 & 0.66666661 & 0.49999933 & 0.24411662 \\
24 & 0.66666664 & 0.49999965 & 0.24436169 \\
\hline
\end{tabular}
\caption{Optimal versus random play by Player II.}
\label{table:random-II-table1}
\end{center}
\end{table}
The last 
column of Table \ref{table:random-II-table1} suggests 
a more precise asymptotic formula for $\poptrand$, stated in the 
following conjecture.

\begin{conj}
The probability $\poptrand$ that Player II wins assuming optimal play by
Player I and random play by Player II satisfies
\begin{equation}
\label{eq:conj-opt-rand}
\frac12-\poptrand\sim \frac{1}{4}\frac{n}{2^{n}} \quad (n\to\infty).
\end{equation}
\end{conj}

\paragraph{Optimal strategy for Player II in the flipped game.}
For the flipped Penney-Ante game in which the player whose string appears \emph{last} wins we determined in Theorem
\ref{thm:flipped-optimal-moves} all optimal strategies for Player I.  
It is natural to ask what the optimal strategies for Player II are in such a flipped game.  
In analogy to the standard Penney-Ante game, a reasonable guess might be
that, given a string selected by Player I, Player II has 
a unique optimal response string consisting of the suffix of length $n-1$
of the string chosen by Player I followed by either an $H$ or a $T$. However,
Table \ref{table:flipped-optimal-responses} shows that, while such strings 
generally do perform well in the flipped game, they are not always
optimal, and that the optimal response string is not always unique.
\begin{table}[H]
	\begin{center}
		\begin{tabular}{|c|c|c|}
			\hline
			String & Best Response String(s) & Probability \\
			\hline
			\hline
		    HHHHH & HHHHH, HHHHT, TTTTT & $1/2$ \\ \hline
		    HHHHT & TTTTT & $31/46$ \\ \hline
		    HHHTH & HHTHH, HHTHT & $2/3$ \\ \hline
		    HHHTT & TTTTT & $31/44$ \\ \hline
		    HHTHH & HHHHH & $7/11$ \\ \hline
		    HHTHT & HTHTH & $10/13$ \\ \hline
		    HHTTH & HTTHT & $9/13$ \\ \hline
		    HHTTT & TTTTT & $31/40$ \\ \hline
			HTHHH & HHHHH & $3/4$ \\ \hline
			HTHHT & THHTT & $17/26$ \\ \hline
			HTHTH & HHHHH & $3/5$ \\ \hline
			HTHTT & THTTH, THTTT & $17/24$ \\ \hline
			HTTHH & HHHHH & $15/22$ \\ \hline
			HTTHT & TTTTT & $31/48$ \\ \hline
			HTTTH & HHHHH & $15/23$ \\ \hline
			HTTTT & TTTTT & $31/32$ \\ \hline
	\end{tabular}
	\end{center}
	\caption{Best response strings, and the corresponding win
	probabilities, in the flipped Penney-Ante
	game for strings of length~$5$.}
	\label{table:flipped-optimal-responses}
\end{table}

Note that in each case in 
Table \ref{table:flipped-optimal-responses}, the optimal response
strings are either of the form $HH\dots H$, $TT\dots T$, 
or consist of the last $n-1$ bits of the
string chosen by Player I followed by an $H$ or $T$.  Computer calculations show that
this pattern persists at least up to $n=10$, thus suggesting 
the following conjecture.

\begin{conj}
\label{conj:flipped-optimal-responses}
Assume Player I chooses  a string $A=a_1\dots a_n$. Then the best
response strings for Player II in the flipped Penney-Ante game are one
or more of the following four strings: 
\begin{equation}
\label{eq:flipped-optimal-responses}
HH\dots H,\quad TT\dots T,\quad a_2\dots a_{n}H,
\quad a_2\dots a_{n}T. 
\end{equation}
\end{conj}


\end{document}